\documentclass[12pt]{amsart}
\usepackage{amsmath,amsfonts,amssymb,amsthm,epsfig,comment}
\usepackage{youngtab,color,tikz}

\usepackage[colorlinks=true, pdfstartview=FitV, linkcolor=blue, citecolor=blue, urlcolor=blue]{hyperref}

\numberwithin{equation}{section}

\newtheorem{Theorem}{Theorem}[section]
\newtheorem{Proposition}[Theorem]{Proposition} 
\newtheorem{Lemma}[Theorem]{Lemma}

\newtheorem{Corollary}[Theorem]{Corollary}

\theoremstyle{remark}
\newtheorem{Remark}[Theorem]{Remark}

\theoremstyle{definition}
\newtheorem{Definition}[Theorem]{Definition}

\newcommand{\SSYT}{\text{SSYT}}

\newcommand{\wt}{\operatorname{wt}}

\newcommand{\g}{\mathfrak{g}}

\theoremstyle{definition}

\newcommand{\yyta}[1]{
\begin{tikzpicture}[scale=0.5]

\draw[line width=0.03cm] (0,0)--(0,1);
\draw[line width=0.03cm] (1,0)--(1,1);

\draw[line width=0.03cm] (0,0)--(1,0);
\draw[line width=0.03cm] (0,1)--(1,1);

\node at (0.5,0.5){#1};

\end{tikzpicture}
}
\newcommand{\yytc}[2]{
\begin{tikzpicture}[scale=0.5]

\draw[line width=0.03cm] (0,0)--(0,4);
\draw[line width=0.03cm] (1,0)--(1,4);

\draw[line width=0.03cm] (0,0)--(1,0);
\draw[line width=0.03cm] (0,1)--(1,1);
\draw[line width=0.03cm] (0,3)--(1,3);
\draw[line width=0.03cm] (0,4)--(1,4);

\node at (0.5,2.25){$\vdots$};
\node at (0.5,0.5){#1};
\node at (0.5,3.5){#2};

\end{tikzpicture}
}

\newcommand{\yytd}[2]{
\begin{tikzpicture}[scale=0.5]

\draw[line width=0.03cm] (0,0)--(0,5);
\draw[line width=0.03cm] (1,0)--(1,5);

\draw[line width=0.03cm] (0,0)--(1,0);
\draw[line width=0.03cm] (0,1)--(1,1);
\draw[line width=0.03cm] (0,4)--(1,4);
\draw[line width=0.03cm] (0,5)--(1,5);

\node at (0.5,2.75){$\vdots$};
\node at (0.5,0.5){#1};
\node at (0.5,4.5){#2};

\end{tikzpicture}
}

\newcommand{\yywa}[1]{
\begin{tikzpicture}[scale=0.5]

\draw[line width=0.03cm] (0,0)--(0,1);
\draw[line width=0.03cm] (1.5,0)--(1.5,1);

\draw[line width=0.03cm] (0,0)--(1.5,0);
\draw[line width=0.03cm] (0,1)--(1.5,1);

\node at (0.75,0.5){#1};

\end{tikzpicture}
}

\newcommand{\yywc}[2]{
\begin{tikzpicture}[scale=0.5]

\draw[line width=0.03cm] (0,0)--(0,4);
\draw[line width=0.03cm] (1.5,0)--(1.5,4);

\draw[line width=0.03cm] (0,0)--(1.5,0);
\draw[line width=0.03cm] (0,1)--(1.5,1);
\draw[line width=0.03cm] (0,3)--(1.5,3);
\draw[line width=0.03cm] (0,4)--(1.5,4);

\node at (0.75,2.25){$\vdots$};
\node at (0.75,0.5){#1};
\node at (0.75,3.5){#2};

\end{tikzpicture}
}

\newcommand{\yywd}[2]{
\begin{tikzpicture}[scale=0.5]

\draw[line width=0.03cm] (0,0)--(0,5);
\draw[line width=0.03cm] (1.5,0)--(1.5,5);

\draw[line width=0.03cm] (0,0)--(1.5,0);
\draw[line width=0.03cm] (0,1)--(1.5,1);
\draw[line width=0.03cm] (0,4)--(1.5,4);
\draw[line width=0.03cm] (0,5)--(1.5,5);

\node at (0.75,2.75){$\vdots$};
\node at (0.75,0.5){#1};
\node at (0.75,4.5){#2};

\end{tikzpicture}
}

\newcommand{\arxiv}[1]{\href{http://arxiv.org/abs/#1}{\tt arXiv:\nolinkurl{#1}}}


\begin{document}

\title[Young tableaux, multisegments, and PBW bases]{Young tableaux, multisegments, and PBW bases}

\author{John Claxton}
\address{Department of Mathematics and Statistics, Loyola University, Chicago, IL}
\email{johnclaxton56@gmail.com}

\author[Peter Tingley]{Peter Tingley$^*$}
\address{Department of Mathematics and Statistics, Loyola University, Chicago, IL}
\email{ptingley@luc.edu}

\begin{abstract}
The crystals for finite dimensional representations of $\mathfrak{sl}_{n+1}$ can be realized using Young tableaux. The infinity crystal on the other hand is naturally realized using multisegments, and there is a simple description of each embedding $B(\lambda) \hookrightarrow B(\infty)$ in terms of these realizations. The infinity crystal is also parameterized by Lusztig's PBW basis with respect to any reduced expression for $w_0$. We give an explicit description of the unique crystal isomorphism from PBW bases to multisegments in the case where $w_0= s_1 s_2 s_3 \cdots s_n s_1 \cdots s_1 s_2 s_1$, thus obtaining simple formulas for the actions of all crystal operators on this PBW basis. Our proofs use the fact that the twists of the crystal operators by Kashiwara's involution also have simple descriptions in terms of multisegments, and a characterization of $B(\infty)$ due to Kashiwara and Saito. These results are to varying extents known to experts, but we do not think there is a self-contained exposition of this material in the literature, and our proof of the relationship between multisegments and PBW bases seems to be new. 
\end{abstract}

%%%%arxiv abstract
\begin{comment}
The crystals for finite dimensional representations of sl(n+1) can be realized using Young tableaux. The infinity crystal on the other hand is naturally realized using multisegments, and there is a simple description of the embedding of each finite crystal into the infinity crystal in terms of these realizations. The infinity crystal is also parameterized by Lusztig's PBW basis with respect to any reduced expression for the longest word in the Weyl group. We give an explicit description of the unique crystal isomorphism from PBW bases to multisegments for one standard choice of reduced expression, thus obtaining simple formulas for the actions of all crystal operators on this PBW basis. Our proofs use the fact that the twists of the crystal operators by Kashiwara's involution also have simple descriptions in terms of multisegments, and a characterization of the infinity crystal due to Kashiwara and Saito. These results are to varying extents known to experts, but we do not think there is a self-contained exposition of this material in the literature, and our proof of the relationship between multisegments and PBW bases seems to be new. 
\end{comment}

\date{}
\maketitle

\thispagestyle{myheadings}
\font\rms=cmr8 
\font\its=cmti8 
\font\bfs=cmbx8

\markright{\its S\'eminaire Lotharingien de
Combinatoire \bfs 73 \rms (2015), Article~B73c\hfill}
\def\thepage{}

\tableofcontents
 
\section{Introduction}

Kashiwara's crystals $B(\lambda)$ are combinatorial objects corresponding to the highest weight representations of a symmetrizable Kac--Moody algebra. Here we will only consider the case when that algebra is $\mathfrak{sl}_{n+1}$. Then the crystal $B(\lambda)$ can be realized as the set of semi-standard Young tableaux of a fixed shape along with some combinatorial operations.

We also consider the crystal $B(\infty)$ for $U^-(\mathfrak{sl}_{n+1})$, which is a direct limit of the $B(\lambda)$ as $\lambda \rightarrow \infty$. 
There is a combinatorial realization of $B(\infty)$ where the underlying set consists of multisegments (i.e.,  collections of ``segments" $[i,j]$ for various $1 \leq i \leq j \leq n$, allowing multiplicity). Importantly for us, the twists of the crystal operators by Kashiwara's $*$-involution are also easy to describe in this realization. 
Furthermore, 
the weak crystal embeddings $B(\lambda) \hookrightarrow B(\infty)$ are easily understood in terms of Young tableaux and multisegments: Each box corresponds to a segment, and the tableau is sent to the collection of the segment corresponding to each box (see Theorem~\ref{thm:embed}). 

There is another realization of $B(\infty)$ which has as its underlying set Lusztig's PBW monomials. Combinatorially, these are recorded by lists of exponents, called Lusztig data, which consist of an integer for each positive root. The construction depends on a choice of reduced expression for the longest word. The crystal operators are defined algebraically, and are somewhat difficult to work with in general.  

The positive roots for $\mathfrak{sl}_{n+1}$ are naturally in bijection with segments: $\alpha_i+\alpha_{i+1} + \cdots + \alpha_j$ corresponds to $[i,j]$. In this way Lusztig data and multisegments are in bijection. 
In most cases this bijection does not seem to have nice properties but, if we work with the 
 reduced expression
\begin{equation}
\label{eq:rex}
w_0= s_1 s_2 \cdots s_n s_1 s_2 \cdots s_{n-1} \cdots  s_1 s_2 s_1,
\end{equation}
we show that it is a crystal isomorphism (see Theorem~\ref{thm:PBW-MS-isom}). 

%First page headline in AmS-LaTeX for S\'eminaire Lotharingien de Combinatoire
%--restoring the headers and pagenumbering
\pagenumbering{arabic}
\addtocounter{page}{1}
\markboth{\SMALL JOHN CLAXTON AND PETER TINGLEY}{\SMALL 
YOUNG TABLEAUX, MULTISEGMENTS, AND PBW BASES}

Much of the current work is to some extent understood by experts. The reduced expression \eqref{eq:rex} has been observed to have nice properties many times (see e.g.\ \cite[\S3.4.3]{Kam}, \cite{BBF}, \cite[\S 5]{Littelmann}), and the connection with Young tableaux has been made (see e.g.\ \cite{BZ}, \cite[Prop. 2.3.13]{M}). The map from tableaux to multisegments has been studied in e.g.\ \cite[\S 8]{BZ} and \cite[\S7]{Zelevinsky}, although there it is not discussed in terms of crystals. Various relationships between the infinity crystal and multisegments have also been observed (see e.g.\ \cite{LTV}). Kashiwara's $*$ involution is well known in the context of multisegments:  it is precisely the famous Zelevinsky multisegment duality first introduced in \cite{oZel,oZel2} (see also \cite{Zelevinsky}). Finally, in the affine $\mathfrak{sl}_{n+1}$ case, the embeddings $B(\Lambda) \hookrightarrow B(\infty)$ are described by Jacon and Lecouvey \cite{JL}; our results from \S\ref{ss:emb} can be derived from their results. 
In fact, much of the literature considers the affine case, partly because it is related to the such important topics as the $p$-adic representations theory of $\mathfrak{gl}_n$ and of certain Hecke algebras (see e.g.\ \cite{BZ2, Vaz}), but this can obscure the simpler finite type case.

Ringel's Hall algebra approach to quantum groups \cite{Rin} can also be used to see some of our results: By Gabriel's theorem~\cite{Gab}, in any finite type, 
$U^-(\g)$ can be identified (as a vector space) with the split Grothendieck group of the category of representations of the quiver obtained by choosing an orientation of the Dynkin diagram. Ringel introduced a product in terms of the representation theory of the quiver that strengthens this relationship, and it was shown in \cite{Lus:can-arise,Rin2}  that the natural basis of the Grothendieck group consisting of isomorphism classes of representations coincides with Lusztig's PBW basis for a reduced word adapted to the orientation. Reineke \cite{Rei} gave an explicit description of the crystal operators acting on the PBW basis in terms of representations of quivers. In the $\mathfrak{sl}_{n+1}$ case, isomorphism classes of representation are naturally indexed by multisegments. Choosing the appropriate orientation of the quiver, Reineke's work implies our results from \S\ref{sec:pbwtoms}. 

In any case, we do not know a self-contained exposition of these results. Our methods are considerably more elementary and combinatorial than most of the references discussed above, and some of our proofs are new.

\subsection{Acknowledgements}
We thank Ben Salisbury, Monica Vazirani and Arun Ram for interesting discussions and for comments on an early draft. We also thank Tynan Greene who did some preliminary work with us in summer 2013. Finally, we thank the anonymous referee for providing some important references. Both authors received partial support from the NSF grant DMS-1265555.

\section{Background}

\subsection{The quantized universal enveloping algebra} \label{sec:quea}

$U_q(\mathfrak{sl}_{n+1})$ is the quantized universal enveloping algebra for $\mathfrak{sl}_{n+1}$. It is an algebra over ${\mathbb C}(q)$  generated by $\{ E_i, F_i, K_i^{\pm 1} \}$, for $1 \leq i \leq n$. Details can be found in e.g.\ \cite{CP}. We will mainly work with $U^-_q(\mathfrak{sl}_{n+1})$, the subalgebra generated by the $F_i$. We first fix some notation
\begin{itemize}
\item $W$ is the Weyl group, and $w_0$ is the longest element in $W$. 

\item $P$ is the weight lattice, and $Q \subset P$ is the root lattice. $P^\vee$ and $Q^\vee$ are the co-weight and co-root lattices. 

\item $\{ \alpha_i\}_{i=1, \ldots, n}, \{ \omega_i \}_{i=1, \ldots, n}$ are the simple roots and fundamental weights respectively; $\{ \alpha_i^\vee\}_{i=1, \ldots, n}$ are the simple co-roots. 

\item $\langle \cdot, \cdot \rangle$ is the pairing between the root lattice and the co-root lattice defined by
$$\langle \alpha_i, \alpha_j^\vee \rangle =
\begin{cases}
2 \quad \text{ if } i-j \\
-1 \quad \text{if } |i-j|=1 \\
0 \quad \text{otherwise}. 
\end{cases}
$$

\item $\displaystyle N= \left(
\begin{array}{c}
n \\
2
\end{array}
 \right)$ is the number of positive roots $\mathfrak{sl}_{n+1}$. 

\item $T_i$ is the algebra automorphism of $U_q(\mathfrak{sl}_{n+1})$ introduced by Lusztig (see \cite[37.1.3]{Lusztig}):
\[
\begin{aligned}
&T_i(F_j) =
\begin{cases}
F_j  \quad i \text{ not adjacent to } j \\
F_j F_i - q F_i F_j \quad i \text{ adjacent to } j \\
-K_j^{-1} E_j \quad i=j.  
\end{cases}
\end{aligned}
\]
\[
\begin{aligned}
&T_i(E_j) =
\begin{cases}
E_j  \quad i \text{ not adjacent to } j \\
E_j E_i - q^{-1} E_i E_j \quad i \text{ adjacent to } j \\
-F_j K_j \quad i=j.  
\end{cases}
\end{aligned}
\]
\[
\begin{aligned}
& T_i(K_j) =
\begin{cases}
K_j  \quad i \text{ not adjacent to } j \\
K_i K_j  \quad i \text{ adjacent to } j \\
K_j^{-1} \quad i=j.  
\end{cases}
\end{aligned}
\]
These $T_i$ define an action of the $n$-strand braid group on $U_q(\mathfrak{sl}_{n+1})$, which means
\[
T_i T_j T_i=T_j T_i T_j \quad \text{if} \quad |i-j|=1, \quad \text{ and}  \qquad T_iT_j=T_iT_i \quad \text{otherwise}.
\]
\end{itemize}

There are embeddings of $U_q(\mathfrak{sl}_n) \hookrightarrow U_q(\mathfrak{sl}_{n+1})$ for all $n$, which just takes the generators of $U_q(\mathfrak{sl}_n)$ to the generators with the same names in $U_q(\mathfrak{sl}_{n+1})$, and these are compatible with the braid group actions.

\subsection{Crystals} \label{ss:crystals}

Here we very briefly introduce Kashiwara's crystals. The following definition is essentially from  \cite[\S7.2]{Kashiwara:1995}, although here we do not allow $\varphi_i,  \varepsilon_i$ to take the value $-\infty$.

\begin{Definition}\label{def:crystal}  An {\bf abstract crystal} is a set $B$ along with functions $\wt \colon B \to P$ (where $P$ is the weight
  lattice), and, for each $i \in I$, $\varepsilon_i, \varphi_i \colon B \to {\mathbb Z}$ and $e_i, f_i: B \rightarrow B \sqcup \{ 0 \}$, such that
  \begin{enumerate}
  \item \label{cry1} $\varphi_i(b) = \varepsilon_i(b) + \langle \wt(b), \alpha_i^\vee \rangle$.
  \item \label{cry2}  If $e_i(b) \neq 0$, $e_i$ increases $\varphi_i$ by 1, decreases $\varepsilon_i$ by 1 and increases $\wt$ by $\alpha_i$.
  \item \label{cry3}  $f_i b = b'$ if and only if $e_i b' = b$.
  \end{enumerate}
We often denote an abstract crystal simply by $B$, suppressing the other data.
\end{Definition}

\begin{Definition}
 A {\bf strict morphism} of crystals is a map between two crystals that commutes with $\wt, e_i, f_i, \varepsilon_i,$ and $\varphi_i$ for all $i$. A {\bf weak morphism} is a map which commutes with all $e_i$, but not necessarily the other structure.  An {\bf isomorphism} of crystals is a strict morphism that has an inverse which is also a strict morphism. 
\end{Definition}

\begin{Remark}
In fact, a weak morphism $\phi: B_1 \rightarrow B_2$ must also have good properties with respect to the other structures. For instance, Definition~\ref{def:crystal} \eqref{cry3} implies that, as long as $b \in B_1$ satisfies $f_i(b) \neq 0$, we have $f_i(\phi(b))= \phi(f_i(b))$. It can however happen that $f_i(b)=0$ but $f_i(\phi(b)) \neq 0$. 
\end{Remark}

\begin{Definition}
\label{def:hwcrystal}
A {\bf highest weight} abstract crystal is an abstract crystal 
which has a distinguished element $b_+$ (the highest weight element) such that 
\begin{enumerate}
\item \label{hwcc1} The highest weight element $b_+$ can be reached from any $b \in B$ by applying a sequence of $e_i$  for various $i \in I$. 
\item \label{hwcc2} For all $b \in B$ and all $i \in I$, $\varepsilon_i(b) = \max \{ n : e_i^n(b) \neq 0 \}$.
\end{enumerate}
  \end{Definition}
Notice that, since a highest weight abstract crystal is necessarily connected, it can have no non-trivial automorphisms. 

The crystals we are interested in here are $B(\infty)$, which is related to $U^-(\mathfrak{sl_{n+1}})$, and $B(\lambda)$, which is related to a highest weight representation of $\mathfrak{sl_{n+1}}$. These are all highest weight abstract crystals. We don't need details of how they are defined; instead we just use the characterization of $B(\infty)$ below, and the explicit description of $B(\lambda)$ in terms of Young tableaux from \S\ref{ss:yt}. 

The following notion is very convenient for us. It is a bit non-standard, but can be found in \cite{TW}. 
\begin{Definition} \label{def:bicrystal}
  A {\bf bicrystal} is a set $B$ with two crystal structures
  whose weight functions agree.  We use the convention of
  placing a star superscript on all data for the second crystal
  structure, so $e_i^*,f_i^*,\varphi^*_i$, etc. 
  An element of a bicrystal is called {\bf highest weight} if it is
killed by both $e_i$ and $e_i^*$ for all $i$.  
\end{Definition}

The following is a rewording of \cite[Proposition~3.2.3]{KS97} designed to make the roles of the usual crystal operators and the $*$-crystal operators more symmetric. See \cite{TW} for this exact statement.

\begin{Proposition} \label{prop:comb-characterization}
Fix a bicrystal $B$. Assume $(B, e_i, f_i)$ and $(B, e_i^*, f_i^*)$ are both highest weight abstract crystals with the same highest weight element $b_+$, where the other data is determined by setting $\wt(b_+)=0$. Assume further that, for all $i \neq j \in I$ and all $b \in B$,
\begin{enumerate}

\item \label{ccc0} $f_i(b), f_i^*(b) \neq 0$. 

\item \label{ccc1} $f_i^*f_j(b)= f_jf_i^*( b)$.

\item \label{ccc2} $\varepsilon_i(b)+\varepsilon_i^*(b)+ \langle \wt(b),
  \alpha_i^\vee \rangle\geq0$ 

\item \label{ccc3} If $\varepsilon_i(b)+\varepsilon_i^*(b)+ \langle \wt(b),
  \alpha_i^\vee \rangle=0$ then $f_i(b) = f_i^*(b)$, 

\item \label{ccc4} If $\varepsilon_i(b)+\varepsilon_i^*(b)+ \langle \wt(b),
  \alpha_i^\vee \rangle \geq 1$ then $\varepsilon_i^*(f_i(b))= \varepsilon_i^*(b)$ and  $\varepsilon_i(f^*_i(b))= \varepsilon_i(b)$.

\item \label{ccc5} If $\varepsilon_i(b)+\varepsilon_i^*(b)+ \langle \wt(b),
  \alpha_i^\vee \rangle \geq 2$ then $f_i f_i^*(b) = f_i^*f_i(b)$. 

\end{enumerate}
Then $(B, e_i, f_i) \simeq (B, e_i^*, f_i^*) \simeq B(\infty)$, and
$e_i^*= *e_i *, f_i^*=*f_i*$, where $*$ is Kashiwara's
involution from \cite[2.1.1]{Kas93}. 
\end{Proposition}

The following is immediate from conditions~\eqref{ccc3}, \eqref{ccc4} and \eqref{ccc5} of Proposition~\ref{prop:comb-characterization}. 

\begin{Remark}
The quantities in Proposition~2.6 have also been studied by Lauda and Vazirani \cite{LV} for related reasons. For instance, there $\varepsilon_i(b)+\varepsilon_i^*(b)+ \langle \wt(b),
  \alpha_i^\vee \rangle$ is called $\text{jump}_i(b)$. 
\end{Remark}

\begin{Corollary} \label{cor:KS-diag}
For any $i \in I$ and any $b \in B(\infty)$ the subset of $B(\infty)$ that can be reached from $b$ by applying sequences of the operators $e_i, f_i, e_i^*, f_i^*$ is of the following form, where the solid and dashed arrows show the action of
$f_i$, the dotted and dashed arrows show the action of $f_i^*$, and the width of the diagram at the bottom is $\langle \wt(b_{top}), \alpha_i^\vee \rangle$ for the top vertex $b_{top}$ (in this example the width is $4$).
\vspace{0.15cm}

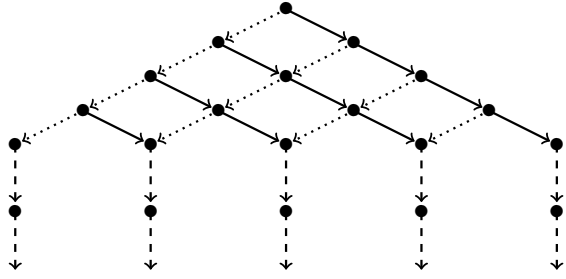
\begin{equation*} \label{ii*-pic}
\setlength{\unitlength}{0.15cm}
\begin{tikzpicture}[xscale=0.45,yscale=0.45, line width = 0.03cm]

\draw node at (10,5) {$\bullet$};

\draw node at (8,4) {$\bullet$};
\draw node at (12,4) {$\bullet$};
\draw node at (6,3) {$\bullet$};
\draw node at (10,3) {$\bullet$};
\draw node at (14,3) {$\bullet$};

\draw node at (4,2) {$\bullet$};
\draw node at (8,2) {$\bullet$};
\draw node at (12,2) {$\bullet$};
\draw node at (16,2) {$\bullet$};
\draw node at (2,1) {$\bullet$};
\draw node at (6,1) {$\bullet$};
\draw node at (10,1) {$\bullet$};
\draw node at (14,1) {$\bullet$};
\draw node at (18,1) {$\bullet$};

\draw node at (2,-1) {$\bullet$};
\draw node at (6,-1) {$\bullet$};
\draw node at (10,-1) {$\bullet$};
\draw node at (14,-1) {$\bullet$};
\draw node at (18,-1) {$\bullet$};

\draw [->, dotted] (10,5)--(8.2,4.1); 
\draw [->, dotted] (8,4)--(6.2,3.1); 
\draw [->, dotted] (12,4)--(10.2,3.1); 
\draw [->, dotted] (6,3)--(4.2,2.1); 
\draw [->, dotted] (10,3)--(8.2,2.1); 
\draw [->, dotted] (14,3)--(12.2,2.1); 

\draw [->, dotted] (4,2)--(2.2,1.1); 
\draw [->, dotted] (8,2)--(6.2,1.1); 
\draw [->, dotted] (12,2)--(10.2,1.1); 
\draw [->, dotted] (16,2)--(14.2,1.1); 
\draw [->] (10,5)--(11.8,4.1); 

\draw [->] (8,4)--(9.8,3.1); 
\draw [->] (12,4)--(13.8,3.1); 
\draw [->] (6,3)--(7.8,2.1); 
\draw [->] (10,3)--(11.8,2.1); 
\draw [->] (14,3)--(15.8,2.1); 
\draw [->] (4,2)--(5.8,1.1); 
\draw [->] (8,2)--(9.8,1.1); 
\draw [->] (12,2)--(13.8,1.1); 
\draw [->] (16,2)--(17.8,1.1); 

\draw [->, dashed] (2,1) --(2,-0.7);
\draw [->, dashed] (6,1) --(6,-0.7);
\draw [->, dashed] (10,1) --(10,-0.7);
\draw [->, dashed] (14,1) --(14,-0.7);
\draw [->, dashed] (18,1) --(18,-0.7);
\draw [->, dashed] (2,-1) --(2,-2.7);
\draw [->, dashed] (6,-1) --(6,-2.7); 
\draw [->, dashed] (10,-1) --(10,-2.7);
\draw [->, dashed] (14,-1) --(14,-2.7);
\draw [->, dashed] (18,-1) --(18,-2.7);

\end{tikzpicture}
\end{equation*}

\noindent Furthermore, for any element $b$, the quantity $\varepsilon_i(b)+\varepsilon_i^*(b)+ \langle \wt(b),
  \alpha_i^\vee \rangle$ counts how many times one must apply $f_i$ (or equivalently $f_i^*$) to reach a dashed line. \qed
\end{Corollary}

\subsection{PBW bases and crystal bases} 

Fix a reduced expression $w_0 = s_{i_1} \cdots s_{i_N}$, and let ${\bf i}$ denote the sequence of indices $i_1, \ldots, i_N$. It is well known that this gives an
ordering of the positive roots of $\mathfrak{sl}_{n+1}$: $\beta_1 = \alpha_{i_1}$ and
$\beta_j = s_{i_1} \cdots s_{i_{j-1}} \alpha_{i_j}$ for $2 \le j \le
N$. Define 
\begin{equation} \label{eq:pbwb-def}
F_{\beta_1} = F_{i_1}, \quad F_{\beta_2}= T_{i_1} F_{i_2}, \quad  \ldots , \quad F_{\beta_j} = T_{i_1} \cdots T_{i_{j-1}} F_{i_j}, \ldots
\end{equation}
where the $T_{i_j}$ are the braid group operators defined in \S\ref{sec:quea}.
As shown by Lusztig \cite[Corollary~40.2.2]{Lusztig},
$B_{\bf i} = \{F_{\beta_N}^{(a_n)} \cdots F_{\beta_1}^{(a_1)} \}$ is a basis for $U_q^-(\g)$ (called the PBW basis). Here 
$F_\beta^{(a)}$ means the quantum divided power
\[F_\beta^{(a)}= F_\beta^a/[a]!\]
where $[a]= q^{-a+1}+q^{-a+3}+ \cdots + q^{a-3} +q^{a-1}$ and $[a]!= [a][a-1] \cdots [2]$. 

Let $\mathcal{A}= {\mathbb C}[q]_0$, the algebra of rational function in $q$ that do not have a pole at $q=0$. By \cite[\S42.1]{Lusztig}, 
$\mathcal{L}= \text{span}_\mathcal{A} B_{\bf i}$ does not depend on the reduced expression ${\bf i}$. This $\mathcal{A}$-module is called the crystal lattice. Furthermore, $B_{\bf i}+q \mathcal{L}$ is a basis for $\mathcal{L}/q\mathcal{L}$ which does not depend on ${\bf  i}$. We denote this basis of  $\mathcal{L}/q\mathcal{L}$ by $B$, and call it the crystal basis. 

\begin{Definition} \label{def:index-basis}
For each reduced expression ${\bf i}$ and each collection of non-negative integers $a_1, \ldots, a_{N}$, let $b^{\bf i}_{a_{1}, \ldots, a_{N}}$ denote the element 
$F_{\beta_N}^{(a_n)} \cdots F_{\beta_1}^{(a_1)}+ q \mathcal{L}$
of the crystal basis $B$. 
\end{Definition}

Definition \ref{def:index-basis} gives a parameterization of $B$ by ${\mathbb N}^N$ for each reduced expression ${\bf i}$ of $w_0$. We now define the crystal operators:

\begin{Definition}  \label{def:PBW-crystal-ops}
Fix $i$. Choose ${\bf i}$ so that $i_1=i$ and $ {\bf i'}$ so that $i'_N= n+1-i$ (so that $\beta_N=\alpha_i$). 
For any $a_1, \ldots, a_N$, define
\[
\begin{aligned}
& f_i(b^{\bf i}_{a_{1}, \ldots, a_{N}})=b^{\bf i}_{a_{1}+1, \ldots, a_{N}}, &&
f_i^*(b^{\bf i'}_{a_{1}, \ldots, a_{N}})=b^{\bf i}_{a_{1}, \ldots, a_{N}+1},
 \\
& e_i(b^{\bf i}_{a_{1}, \ldots, a_{N}}) = \begin{cases}
b^{\bf i}_{a_{1}-1, \ldots, a_{N}} \quad \text{ if } a_1 \geq 1 \\
\emptyset \quad \text{ otherwise},
\end{cases} &&
e_i^*(b^{\bf i}_{a_{1}, \ldots, a_{N}}) = \begin{cases}
b^{\bf i'}_{a_{1}, \ldots, a_{N}-1} \quad \text{ if } a_N \geq 1 \\
\emptyset \quad \text{ otherwise},
\end{cases}
 \\
& \varepsilon_i(b^{\bf i}_{a_{1}, \ldots, a_{N}})= a_1,
&&
\varepsilon_i^*(b^{\bf i'}_{a_{1}, \ldots, a_{N}})= a_N.
\end{aligned}
\]

\end{Definition}

\begin{Theorem}[{\cite[Theorem 4.1.2 and its proof]{Sai}}]
$B$ is the crystal basis for $U^-_q(\mathfrak{sl}_{n+1})$, as defined by Kashiwara \cite{Kashiwara:1995}, and the $e_i, e_i^*, f_i, f_i^*$ defined above are the crystal operators. In particular, $B$ along with the operations from Definition~\ref{def:PBW-crystal-ops} is a realization of the crystal $B(\infty)$ from \S\ref{ss:crystals}.
\end{Theorem}
 
Fix $i$ and a reduced expression for $w_0$ of the form ${\bf i}=(i, i_2, \cdots, i_N)$. Then ${\bf i'}= (i_2, i_3, \ldots i_N, n+1-i)$ is also a reduced expression. It is clear from the definitions that $T_i^{-1}$ gives the bijection 
$\{ b_{0,a_2, \ldots a_N } \longrightarrow   b'_{a_2, \ldots a_N, 0} \}$
between the subset of those $b \in B(\infty)$ where $\varepsilon_{i}(b)=0$, and those $b \in B(\infty)$ where $\varepsilon_{i}^*(b)=0$. 
Define $\tau:B \rightarrow \{ b \in B: \varepsilon_{i}(b)=0 \}$ by
 $\tau:b^{\bf i}_{a_1,a_2, \ldots a_N } \mapsto b^{\bf i}_{0,a_2, \ldots a_N }$ and let $\sigma_i=T_i^{-1} \circ \tau$.

\begin{Proposition} \label{prop:Saito-formula}
For any $b \in B(\infty)$,
\begin{equation} \label{eq:orf}
\sigma_i(b)= (e_i^*)^{\text{max}}f_i^L b \text{ for any } L \geq \varepsilon_i(b)+\varepsilon_i^*(b) +\langle \wt(b), \alpha_i^\vee\rangle.
\end{equation}
\end{Proposition}

\begin{proof}
It is immediate from the definitions that it suffices to consider the case $\varepsilon_i(b)=0$. Then by \cite[Corollary~3.4.8]{Sai}, 
\begin{equation} \label{eq:sf}
\sigma_i(b)= f_i^{\varphi_i^*(b)}(e_i^*)^{\varepsilon_i^*(b)}(b).
\end{equation}

First assume  $L=  \varepsilon_i(b)+\varepsilon_i^*(b) +\langle \wt(b), \alpha_i^\vee\rangle$, and refer to 
the diagram from Corollary~\ref{cor:KS-diag} showing the part of the crystal reachable from $b$ by applying $e_i,f_i, e_i^*, f_i^*$. By definition, $\varphi_i^*(b)= \langle \wt(b), \alpha_i^\vee \rangle +\varepsilon_i^*(b)$ which is equal to $\langle \wt(b), \alpha_i^\vee \rangle +\varepsilon_i(b)+\varepsilon_i^*(b)$ since $\varepsilon_i(b)=0$. By Corollary~\ref{cor:KS-diag} we see that $\varphi^*_i(b)$ is the number of times one must apply $f_i^*$ to $b$ to reach the dashed line. 
Thus \eqref{eq:sf} takes $b$, applies $e_i^*$ until it gets to the vertex, then applies $f_i$ exactly $\varphi^*_i$ times. On the other hand, \eqref{eq:orf} takes $b$, applies $f_i$ exactly $\varphi^*_i$ times to just reach the dashed lines, then applies $e_i^*$ the maximal number of times, thus reaching the top right boundary of the picture. Tracing this through, they agree. 

If $L> \varepsilon_i(b)+\varepsilon_i^*(b) +\langle \wt(b), \alpha_i^\vee\rangle$ then \eqref{eq:orf} just applies some extra $f_i$ which move along dashed lines, then some extra $e_i^*$, which undoes this, and so the result does not change. 
\end{proof}
The following key lemma will allow us to perform induction on rank.

\begin{Lemma} \label{lem:braid} 
When $w_0=s_1s_2 \cdots s_ns_1 \cdots s_{n-1} \cdots s_1s_2s_1$, the corresponding order on positive roots is
$$\beta_1= \alpha_1, \beta_2= \alpha_1+\alpha_2, \ldots, \beta_{n}= \alpha_1 + \cdots+\alpha_n, \beta_{n+1}= \alpha_2, \ldots, \beta_N= \alpha_n.$$  
Furthermore, 
\begin{align*}
\sigma_n \cdots \sigma_2\sigma_1
& \left(F_{\alpha_1}^{(a_1)} \cdots F_{(\alpha_1+ \dots + \alpha_n)}^{(a_n)}F_{\alpha_2}^{(a_{n+1})} \cdots F_{(\alpha_2+ \dots + \alpha_n)}^{(a_{2n-1})} \cdots F_{\alpha_n}^{(a_N)}\right) \\
&= F_{\alpha_1}^{(a_{n+1})} \cdots F_{(\alpha_1+ \dots + \alpha_{n-1})}^{(a_{2n-1})}F_{\alpha_2}^{(a_{2n})} \cdots F_{\alpha_{n-1}}^{(a_N)}.
\end{align*}
\end{Lemma}
\begin{proof}
This is a simple calculation using \eqref{eq:pbwb-def}.
\end{proof}

\section{Multisegment and Young tableau realizations} \label{sec-msyt}

\subsection{Multisegment realization of $B(\infty)$} \label{ss:MS}

We now define multisegments and their crystal structure, and prove that they realize $B(\infty)$. 
This is essentially the same as the realization discussed in \cite[\S4.1]{Savage:2006}, although our proof is quite different, and there the term multisegment is not used. We have changed terminology to match \cite{JL}, where they consider the affine case. The term multisegment has also been used as we use it in e.g.\ \cite{Zelevinsky}. Our realization is also very similar to the one constructed in terms of marginally large tableaux in \cite{HL} (see also \cite{LS}), although we note that the $*$ operators are a bit easier to describe using the setup here.

\begin{Definition} \label{def:multisegment}
\begin{enumerate}
\item \upshape  A \textit{segment} is an interval $[i, j]$ with $i,j \in \mathbb{Z}$, $1 \leq i \leq j$.
\item A \textit{multisegment} is a finite set of segments, allowing multiplicity. 
\item Given a multisegment $M$, $M_{i,j}$ is the multiplicity of $[i,j]$ in $M$.
\item $MS_n$ is the set of all multisegments where all segments $[i,j]$ have $j \leq n$.
\item The \textit{height} of a segment, $[i,j]$ is $j-i+1$. 
\item The \textit{size} $|M|$ of a multisegment $M$ is the sum over all segments of their heights. 
\end{enumerate}
\end{Definition}

We will represent a segment $[i,j]$ with a columns of boxes containing the integers $i$ to $j$. For example, $[3,5]$ will be drawn as
\[ \young(5,4,3) \;\; . \]
The size of a multisegment is the total number of boxes if you draw all the segments. 

\begin{Definition} \label{def:strings}
Given $M \in MS_n$, $S_i(M)$ is the string formed as follows (see \S\ref{ss:example} for examples):
\begin{enumerate}
\item[1.] \upshape Order the segments of $M$ from left to right, first in increasing order of height, then by largest to smallest bottom entry. 
\item[2.] Place a ``)" above each $[h,i-1]$ segment, and a ``(" above each $[h,i]$ segment. 
\end{enumerate}
$S_i^*(M)$ is the string formed as follows:
\begin{enumerate}
\item[1.] \upshape Order the segments of $M$ from left to right, first by shortest to tallest, then by \textit {smallest to largest} bottom entry.
\item[2.] Place a ``)" below each $[i+1,j]$ segment, and a ``(" below each $[i,j]$ segment. 
\end{enumerate}
{\upshape In these strings, we say brackets ``(" and ``)" are {\it canceled} if the ``(" is directly to the left of ``)", or if the only brackets between them are canceled. $uc_i(M)$ and $uc_i^*(M)$ are the strings formed from $S_i(M)$ and $S_i^*(M)$ by deleting canceled brackets. }  
\end{Definition} 

For $M \in MS_n$ and $1 \leq i \leq n$, define
\begin{align*}
f_i(M) \hspace{-0.1cm} :=& \begin{cases} (M \backslash \{[h,i-1]\}) \cup \{[h,i]\} & \hspace{-0.25cm} \text{if the right-most ``)" in } uc_i(M) \text{ is above an } [h,i-1] \\
M \cup \{ [i,i] \} &\hspace{-0.25cm}\text{if there are no ``)" in } uc_i(M), \end{cases} \\
e_i(M)  \hspace{-0.1cm} :=& \begin{cases} (M \backslash \{[h,i]\}) \cup \{[h,i-1]\} & \hspace{-0.25cm} \text{if the left-most ``(" in } uc_i(M) \text{ is above an } [h,i],  h \neq i  \\
M\backslash \{[i,i]\} &\hspace{-0.25cm} \text{if the left-most ``(" in } uc_i(M) \text{ is above a } [i,i] \\
0 &\hspace{-0.25cm}\text{if there are no ``(" in } uc_i(M), \end{cases} 
\end{align*}
\begin{align*}
f_i^*(M)  \hspace{-0.1cm} :=& \begin{cases} (M \backslash \{[i+1,j]\}) \cup \{[i,j]\} &\hspace{-0.25cm}\text{if the right-most ``)" in } uc_i^*(M) \text{ is below a } [i+1,j] \\
M \cup \{ [i,i] \} &\hspace{-0.25cm}\text{if there are no ``)" in } uc_i^*(M), \end{cases} \\
e_i^*(M)  \hspace{-0.1cm} :=& \begin{cases} (M \backslash \{[i,j]\}) \cup \{[i+1,j]\} &\hspace{-0.25cm}\text{if the left-most ``(" in } uc_i^*(M) \text{ is below a } [i,j],  i \neq j  \\
M\backslash \{[i,i]\} &\hspace{-0.25cm}\text{if the left-most ``(" in } uc_i^*(M) \text{ is above a } [i,i] \\
0 &\hspace{-0.25cm}\text{if there are no ``(" in } uc_i^*(M). \end{cases} \\
\end{align*}

\begin{Remark}\label{remark:sym}
\upshape There is clearly symmetry between the unstarred and starred operators above. To make this precise, consider the map $\text{Flip}: MS_n \rightarrow MS_n$ which flips every segment over and re-indexes by $1 \leftrightarrow n, 2 \leftrightarrow n-1$, and so on. This sends $S_i(M)$ to $S_{n-i+1}^*(M)$, and so interchanges the action of each $f_i$ with the action of $f^*_{n+1-i}$; in fact, this is the composition of Kashiwara's $*$ involution with the Dynkin diagram automorphism. We will appeal to this symmetry often later on. 
\end{Remark}

For any multisegment $M$, define 
\begin{equation*}
\wt(M) := -\sum\limits_{i=1}^n (\text{\# of } i \text{ boxes})\alpha_i,
\end{equation*}  
and notice that, and all $i$, 
\begin{equation*}
\langle \wt(M), \alpha_i^\vee \rangle = (\text{\# of } i-1 \text{ boxes}) 
\;+\;(\text{\# of } i+1 \text{ boxes}) 
\;-\; 2(\text{\# of } i \text{ boxes}).
\end{equation*}

\begin{Proposition} \label{prop:MS-crystal}
$(MS_n, e_i, f_i, e_i^*, f_i^*)$, with the weight function defined above and additional data determined as in Definition~\ref{def:hwcrystal}, is a highest weight bicrystal with the highest weight element the empty multisegment $\emptyset$. 
\end{Proposition}
\begin{proof}
We need to check that both structures satisfy the axioms in Definitions~\ref{def:crystal} and \ref{def:hwcrystal}; we present the proof only for $(MS_n, e_i, f_i)$, since the arguments for the other structure are the same. \ref{def:crystal}\eqref{cry1}
and \ref{def:hwcrystal}\eqref{hwcc2} are true by definition, and \ref{def:crystal}\eqref{cry2} easily follows from the definitions of $\varepsilon_i, \varphi_i,$ and $\wt$. For \ref{def:crystal}\eqref{cry3} consider how applying $f_i$ affects $S_i$: it either changes $[h,i-1] \to [h,i]$ and a ``)" to a ``(" immediately to the right, or adds an $[i,i]$ and a ``(" as far left as possible. Removing canceled brackets, we have: 
\[ uc_i(M): \quad \dots ))){\color{red} )} \quad ((( \dots \]
\[ uc_i(f_i(M)): \quad \dots ))) \quad {\color{green} (}((( \dots \]
where the red ``(" has become the green ``)."
So $e_i$ will act on the green bracket of $uc_i(f_i(M))$, reversing the effect of $f_i$. \par 

Since multisegments are finite, to establish \ref{def:hwcrystal}\eqref{hwcc1} it suffices to show that, for all non-empty $M \in MS_n$, there is some $e_k$ such that $e_k(M) \neq 0$. The segments are ordered the same way in all of the strings $S_i(M)$, so they all have the same right-most segment, call it $[j,k]$. Clearly $uc_k$ has an uncanceled ``(", and hence $e_k(M) \neq 0$. 
\end{proof}

The next few lemmas are needed to prove that this bicrystal satisfies the conditions in Proposition~\ref{prop:comb-characterization}. The first one explains what $\varepsilon_i(M)+\varepsilon_i^*(M)+ \langle \wt(M), \alpha_i^\vee \rangle$ means in terms of brackets.

\begin{Lemma}\label{lem:bracket-count}
Given a multisegment $M \in MS_n$, for $1 \leq i \leq n$ let 
\begin{align*}
ur_i(M)&:= \text{ the number of uncanceled ``)" in } S_i(M) \\
ur_i^*(M)&:= \text{ the number of uncanceled ``)" in } S_i^*(M)
\end{align*}
Then 
\begin{equation} \label{ind-step lem:bracket-count}
\varepsilon_i(M)+\varepsilon_i^*(M)+ \langle \wt(M), \alpha_i^\vee \rangle = ur_i(M)+ur_i^*(M).
\end{equation}
\end{Lemma}

\begin{proof}
We proceed by induction on $|M|$, the base case $M=\emptyset$ being obvious. So, fix $M \neq \emptyset$ and assume the result holds for all smaller multisegments. $MS_n$ is a highest weight crystal by Proposition~\ref{prop:MS-crystal}, so $M= f_j(M')$ for some $j$ and some $M'$ with $|M'|=|M|-1$, and by the inductive hypothesis the result holds for $M'$. We consider several cases. 

\medskip
{\sc Case 1}: $|i-j|>1$. Then  $\langle wt(M), \alpha_i^\vee \rangle=\langle wt(M'), \alpha_i^\vee \rangle$ and $S_i(M)=S_i(M')$, so $\varepsilon_i$ and $ur_i$ are unchanged. If $f_j$ changes $[i,j-1] \rightarrow [i,j]$ or $[i+1,j-1] \rightarrow [i+1,j]$, these segments and their brackets will shift to the right in $S_i^*$. This shift will either leave $ur_i^*$ unchanged, or change it by 1. We will use the notation $\uparrow 1$ or $\downarrow 1$ to record this change, so for example $ur_i \uparrow 1$ means $ur_i(M) = ur_i(M') +1$.
\begin{itemize}
\item If $ur_i^*$ is unchanged, so is $\varepsilon_i^*$.
\item If $ur_i^* \uparrow 1$, then a ``(" changed from canceled to uncanceled, so $\varepsilon_i^* \uparrow 1$.
\item If $ur_i^* \downarrow 1$, then $\varepsilon_i^* \downarrow 1$.
\end{itemize} \par

\smallskip
{\sc Case 2}: $j=i-1$. Then $\langle wt(), \alpha_i^\vee \rangle \uparrow 1$. The string $S_i^*$ is not affected, so $\varepsilon_i^*$ and $ur_i^*$ are unchanged. 
One new ``)" is created in $S_i$, which either increases $ur_i$ by 1 or leaves it unchanged.  
\begin{itemize}
\item If $ur_i$ is unchanged, then a ``(" that was previously uncanceled must be canceled. So $\varepsilon_i \downarrow 1$.
\item If $ur_i \uparrow 1$, then no new ``(" are canceled, so $\varepsilon_i$ is unchanged.
\end{itemize}

\smallskip
{\sc Case 3}: $j=i+1$. Then $\langle wt(), \alpha_i^\vee \rangle \uparrow 1$. 
\begin{itemize}
\item If $f_{i+1}$ adds a $[i+1, i+1]$, this has no affect on $S_i$, so $\varepsilon_i$ and $ur_i$ are unchanged. The fact that $f_{i+1}$ acted this way implies that all ``)" in $S_{i+1}(M')$ are canceled. In particular, $M'_{i+1,i+1} \geq M'_{i,i}$, so adding an $[i+1,i+1]$ must create an \textit{uncanceled} ``)" in $S^*_i$. Hence $\varepsilon_i^*$ is unchanged and $ur_i^* \uparrow 1$. 
\item If $f_{i+1}$ changes $[i,i] \rightarrow [i,i+1]$, this shifts a ``(" to the right in $S_i^*$. As in the case $|i-j|>1$, $ur_i^*$ and  $\varepsilon_i^*$  are affected in the same way. In $S_i$, a ``(" is removed. This can either increase $ur_i$ or leave it fixed.
\begin{itemize}
\item If $ur_i$ is fixed, there is one less uncanceled ``(", so $\varepsilon_i \uparrow 1$. Also, $\langle \wt(),\alpha_i^\vee \rangle \uparrow 1$. 
\item If $ur_i \uparrow 1$, then $\varepsilon_i$ is fixed. So $ur_i \uparrow 1$ and $\langle wt(), \alpha_i^\vee \rangle \uparrow 1$.
\end{itemize}
\item If $f_{i+1}$ changes $[h,i] \rightarrow [h,i+1]$ for $h \neq i$ this has no effect on $S_i^*$, and the same effect on $S_i$ as in the previous case.
\end{itemize} 

\smallskip
{\sc Case 4}: $j=i$. Then $\langle wt(), \alpha_i^\vee \rangle\downarrow 2$ and $\varepsilon_i \uparrow 1$.
\begin{itemize}
\item If $f_i$ adds an $[i,i]$, the $ur_i(M')=ur_i(M)=0$ so $ur_i$ is unchanged. In $S_i^*$, a ``(" is added, this either decreases $ur_i^*$ or leaves it unchanged.
\begin{itemize}
\item If $ur_i^*$ is unchanged then there is a new uncanceled ``(" so $\varepsilon_i^* \uparrow 1$.
\item If $ur_i^* \downarrow 1$, then $\varepsilon_i^*$ is unchanged. 

\end{itemize}
\item If $f_i$ changes $[h,i-1] \rightarrow [h,i]$, this has no affect on $S_i^*$, so $\varepsilon_i^*$ and $ur_i^*$ are unchanged. In $S_i$, an uncanceled ``)" is changed to a ``(", so  $ur_i \downarrow 1$.
\end{itemize}
In all cases the two sides of \eqref{ind-step lem:bracket-count} change by the same amount.
\end{proof}

\begin{Lemma} \label{lem:uri} Let $M \in MS_n$ and $i \neq j$. For each $h$ let $ur_{i;h}(M)$ be the number of ``)" in $uc_i(M)$ corresponding to segments of height $h$.
\begin{enumerate}
\item \label{ddd1} If $f_i^*$ applied to $M$ acts on an $[i+1, j-1]$ segment and $ur_{j;j-i-1}(M) \neq0$, then 
\begin{align}
ur_{j;j-i-1}(f_i^*(&M)) = ur_{j;j-i-1}(M)-1 \label{lem:uri.1} \\
ur_{j;j-i}(f_i^*(&M)) = ur_{j;j-i}(M)+1  \label{lem:uri.2} 
\end{align}
\item \label{ddd2} If $f_i^*$ applied to $M$ adds an $[i,i]$, then
$
ur_{i+1;1}(f_i^*(M)) = ur_{i+1;1}(M) +1 
$.
\item \label{ddd3} In all other cases, $ur_{j;h}(f_i^*(M))=ur_{j;h}(M)$.
\end{enumerate}
\end{Lemma}

\begin{proof}

{\bf \eqref{ddd1}:} $f_i^*$ acts on a segment of height $j-i-1$, and since $ur_{j;j-i-1}(M) \neq0$, this moves an \textit{uncanceled} ``)" in $S_j(M)$ to the right, and \eqref{lem:uri.1} is clear. For \eqref{lem:uri.2}, we need to show that the new $[i,j-1]$ is uncanceled in $S_j(f_i^*(M))$. Since the bracket was originally uncanceled, it can only be canceled by a segment of height $j-i$. But $ur_{j,j-i-1}(M) \neq 0$ implies that this cannot happen, so \eqref{lem:uri.2} holds.

{\bf \eqref{ddd2}:} Since $f_i^*$ adds an $[i,i]$, all ``)"in $S_i^*(M)$ are canceled. 
In particular, since $[i+1,i+1]$ segments can only be canceled by $[i,i]$ segments, $M_{i+1,i+1} \leq M_{i,i}$. 
So in $S_{i+1}(f_i^*(M))$, the new $[i,i]$ corresponds to an uncanceled ``)."

{\bf \eqref{ddd3}:} This breaks up into four cases
\begin{itemize}

\item If $f_i^*$ acts on a segment $[i+1, h]$ for $h \neq i-1,i$ this clearly has no effect on $S_i$ or $uc_i$. 

\item If $f_i^*$ acts on an $[i+1, j-1]$ segment and $ur_{j;j-i-1}(M)=0$, then a \textit{canceled} ``)" is moved to the right in $S_j$, and since it doesn't move past any other ``)", it's not hard to see that it remains canceled, so this has no effect on $ur_j$. \par

\item  If $f_i^*$  acts on an $[i+1, j]$ segment, this shifts a ``(" to the right. 
The only way this can change $ur_j$ is if an $[i,j-1]$ changes from canceled to uncanceled. But the fact that $f_i^*$ acts on an $[i+1,j]$ implies that $M_{i,j-1} < M_{i+1,j}$; so even after the shift, all $[i,j-1]$ segments in $S_j$ are canceled. \par 

\item If $f_i^*$ adds an $[i,i]$, this has no effect on $S_j$ (so certainly it has no effect on $ur_j$), unless $j=i+1$. Clearly $ur_{i;h}$ cannot be affected for $h>1$, and $h=1$ is case \eqref{ddd2}. 
\end{itemize}
\end{proof}

\begin{Lemma}\label{lem:ccc2}
$f_jf_i^*(M)=f_i^*f_j(M)$ when $i\neq j$.
\end{Lemma}

\begin{proof}
For this equality to break, $f_i^*$ must change the length of the segment that $f_j$ acts on, or vice-versa. By symmetry (see Remark~\ref{remark:sym}), we can assume we are in the first case. By Lemma~\ref{lem:uri}, there is only one way for this to happen: 
 $f_i^*$ must act on $M$ by changing a $[i+1,j-1]$ to $[i, j-1]$ (if $j=i+1$, this is interpreted as adding a new $[i]$), $f_j$ must act on $f_i^*(M)$ by changing that $[i, j-1]$ to an $[i,j]$, and $f_j$ must act on $M$ by changing a different segment, which looking at the brackets in $f_i^*(M)$ must necessarily be a shorter segment. In fact, $f_j$ must act on M by changing a $[i+1,j-1]$, since if there are no uncanceled brackets corresponding to segments of that length in $S_i(M),$ moving one of them to the right past some ``(" cannot create an uncanceled bracket. From this one can quickly see that $f_jf_i^*(M)$ and $f_i^*f_j(M)$ both change an $[i+1,j-1] $ to $ [i,j]$, so they agree.  
\end{proof}

\begin{Proposition} \label{prop:MS-B(inf)}
$MS_n$ along with the operators $e_i,f_i, e_i^*, f_i^*$ is a realization of $B(\infty)$ as a bicrystal, where the second crystal structure is the twist of the first by Kashiwara's involution. 
\end{Proposition}
\begin{proof}
We need to check that $MS_n$ satisfies the conditions in Proposition~\ref{prop:comb-characterization}. 

{\bf (i):} This is clear from the definitions of $f_i$ and $f_i^*$. \par
{\bf (ii):} This is Lemma~\ref{lem:ccc2} \par
{\bf (iii):} This is clear from Lemma~\ref{lem:bracket-count}. \par
{\bf (iv):} By Lemma~\ref{lem:bracket-count}, if $\varepsilon_i(M)+\varepsilon_i^*(M)+ \langle \wt(M), \alpha_i^\vee \rangle=0$, there are no uncanceled ``)" in either $S_i(M)$ or $S_i(M)$, so both $f_i$ and $f_i^*$ add an $[i,i]$. \par
{\bf (v):} By Lemma~\ref{lem:bracket-count}, if $\varepsilon_i(M)+\varepsilon_i^*(M)+ \langle \wt(M), \alpha_i^\vee \rangle \geq 1$, there is an uncanceled ``)" in either $S_i(M)$ or $S_i^*(M)$. Without loss of generality, assume it is in $S_i(M)$. So $f_i$ applied to $M$ must change $[h,i-1] \rightarrow [h,i]$ for some $h \leq i-1$. This has no affect on $S_i^*$, so $\varepsilon_i^*(f_i(M)) =  \varepsilon_i^*(M)$. If there is also an uncanceled ``)" in $S_i^*(M)$, the same argument shows that $\varepsilon_i(f_i^*(M)) =  \varepsilon_i^*(M)$. If there are no uncanceled ``)" in $S_i^*$, then $f_i^*$ adds an $[i,i]$. This creates a ``(" all the way to the left in $S_i(f_i^*(M))$. Since we assumed there was some uncanceled ``)" in $S_i(M)$, it must cancel the new ``(", and so $\varepsilon_i(f_i^*(M)) =  \varepsilon_i(M)$. \par
{\bf (vi):} By Lemma~\ref{lem:bracket-count}, if $\varepsilon_i(M)+\varepsilon_i^*(M)+ \langle \wt(M),
  \alpha_i^\vee \rangle \geq 2$, there are at least 2 uncanceled ``)" between $S_i(M)$ and $S_i^*(M)$. If there is at least one in each string, $f_i$ and $f_i^*$ add to the tops and bottoms of segments respectively, so neither operator affects the other string, from which it is clear that $f_if_i^*(M)=f_i^*f_i(M)$. 
  
 Otherwise, one string has no uncanceled ``)" and the other has at least 2. Assume without loss of generality that $ur_i(M)=0$ and $ur_i^*(M) \geq 2$. As in (v), $S_i(f_i^*(M))=S_i(M)$, so $f_i$ acts  on segments of the same length in $M$ and $f_i^*(M)$. Since $ur_i(M)=0$ we see that $f_i(M) = M \cup \{ [i,i] \}$, so this creates a ``(" all the way to the left in $S_i^*(f_i(M))$, which cancels a ``)". But since there are at least 2 uncanceled ``)" in $S_i^*(M)$, the right-most one is still uncanceled in $S_i^*(f_i(M))$. Therefore $f_i^*$ acts  on segments of the same length in $M$ and $f_i(M)$. It follows that $f_i f_i^*(M)= f_i^* f_i (M)$.  
\end{proof}

\subsection{Young tableau realization of $B(\lambda)$} \label{ss:yt}
Recall that a partition $\lambda = (\lambda_1, \lambda_2, \cdots, \lambda_k)$ is a non-increasing sequence of positive integers. The size of $\lambda$ is $|\lambda|= \lambda_1+ \cdots + \lambda_k$. 
To each partition $\lambda$ and $n \geq k-1$, we associate a weight for $\mathfrak{sl}_{n+1}$ by 
$$\lambda \rightarrow (\lambda_1-\lambda_2) \omega_1+ \cdots + (\lambda_{k-1}-\lambda_{k} )\omega_{k-1}+ \lambda_k \omega_k,$$ 
where in the case $n=k-1$ we drop the last term. As usual we will often simply write $\lambda$ to mean this associated weight when the meaning is clear from context.

We associate to a partition its {\bf Young diagram}, which consists of the $|\lambda|$ boxes, arranged as a row of length $\lambda_1$ above a row of length $\lambda_2$, etc. 
A semi-standard Young tableau of shape $\lambda$ for $\mathfrak{sl}_{n+1}$ is a filling of the Young diagram of $\lambda$ with the numbers $\{0,\ldots,n\}$, which is weakly increasing along rows and strictly increasing down columns. See Figure~\ref{fig:yt}. Denote the set of all such tableaux of a fixed shape by $\SSYT_n(\lambda)$. 

Define operators $f_i$ on $\SSYT_n(\lambda)$ for $1 \leq i \leq n$ as follows: 
 $f_i$ will change a $i-1$ to a $i$, or else send the tableau to $0$. To determine which $i-1$ changes, place a  ")" above each column that contains a $i-1$ but no $i$, and a ``(" above each column that contains a $i$ but no $i-1$. Cancel the brackets as usual. $f_i$ changes the $i-1$ corresponding to the right-most uncanceled ``)" to a $i$, or, if there is no uncanceled ``)," then $f_i(b)=0$. As with multisegments, $e_i$ is calculated with the same string of brackets, and changes the left-most uncanceled $i$ to an $i-1$. 

\begin{figure}
\[
\begin{tikzpicture}[scale=0.5]
\node [anchor=south west] at (0,0){$\young(000112234,1122334,234,4)$};
\node at (1.8,4.8){)};
\node at (2.65,4.8){(};
\node at (4.6,4.8){)}; 
\node at (5.5,4.8){(}; 
\node at (6.4,4.8){(}; 
\end{tikzpicture}
\]
\caption{\label{fig:yt} A Young tableau $b$ of shape $(9,7,3,1)$. The string of brackets $S_2(b)$ is shown. $f_2(b)$ is obtained by changing the right-most ``1" on the second row to a ``2".   }
\end{figure}

The following is well known. It can be found in a slightly different form in e.g.\ \cite[\S 5]{Kashiwara:1995}.
\begin{Theorem} 
$\SSYT_n(\lambda)$ with the operators defined above is a realization of the $\mathfrak{sl}_{n+1}$-crystal $B(\lambda)$.
\qed
\end{Theorem}

\subsection{Embeddings $B(\lambda) \hookrightarrow B(\infty)$} \label{ss:emb}

The following map is essentially the map from\break tableaux to Kostant partitions from \cite{LS}, although presented a bit differently (and, as a warning, the term ``segment" is used differently there). The same map is also studied in \cite[\S 7]{Zelevinsky}, although the fact that it is a crystal morphism is not explicitly discussed there.

\begin{Definition} \label{def:YT-MS}
For a Young tableau $b \in \SSYT_n(\lambda)$ define the corresponding multisegment $M_b$ to be the collection containing the segment $[i,j]$ for each $j \geq i$ in the $i^{th}$ row of $b$. See Figure~\ref{fig:embed}. 
\end{Definition}

\newsavebox{\btab}
\savebox{\btab}{
\begin{tikzpicture}[scale=0.5]
\node [anchor=south west] at (0,0){$\young(0123,123,23)$};
\node [anchor=south west] at (3.05,3){(};
\node [anchor=south west] at (.35,3){)};
\end{tikzpicture}}

\newsavebox{\ebtab}
\savebox{\ebtab}{
\begin{tikzpicture}[scale=0.5]
\node [anchor=south west] at (0,0){$\young(0122,123,23)$};
\end{tikzpicture}}

\newsavebox{\bseg}
\savebox{\bseg}{
\begin{tikzpicture}[scale=0.5]
\node [anchor=south west] at (1,0){$\young(3)$};
\node [anchor=south west] at (1.25,3){(};
\node [anchor=south west] at (2.5,0){$\young(2)$};
\node [anchor=south west] at (2.75,3){)};
\node [anchor=south west] at (4,0){$\young(1)$};
\node [anchor=south west] at (5.5,0){$\young(3,2)$};
\node [anchor=south west] at (5.55,3){(};
\node [anchor=south west] at (7,0){$\young(2,1)$};
\node [anchor=south west] at (7.25,3){)};
\node [anchor=south west] at (8.5,0){$\young(3,2,1)$};
\node [anchor=south west] at (8.75,3){(};
\end{tikzpicture}
}

\newsavebox{\ebseg}
\savebox{\ebseg}{
\begin{tikzpicture}[scale=0.5]
\node [anchor=south west] at (1,0){$\young(3)$};
\node [anchor=south west] at (2.5,0){$\young(2)$};
\node [anchor=south west] at (4,0){$\young(1)$};
\node [anchor=south west] at (5.5,0){$\young(3,2)$};
\node [anchor=south west] at (7,0){$\young(2,1)$};
\node [anchor=south west] at (8.5,0){$\young(2,1)$};
\end{tikzpicture}
}

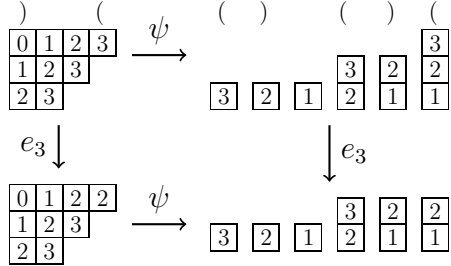
\begin{figure}
\[
\begin{tikzpicture}[scale=.45, every node/.style={scale=.75}]
\node at (0,5) (1) {\usebox{\btab}};
\node at (8,5) (2) {\usebox{\bseg}};
\node at (0,0) (3) {\usebox{\ebtab}};
\node at (8,0) (4) {\usebox{\ebseg}};

\path[->, thick, every node/.style={font=\sffamily\normalsize}]
    (1) edge node [above] {$\psi$} (2)
    (3) edge node [above] {$\psi$} (4)
    (1) edge node [left] {$e_3$} (3)
    (2) edge node [right] {$e_3$} (4);
\end{tikzpicture}
\]
\caption{\label{fig:embed} Moving from Young tableaux to multisegments and applying $e_3$.}
\end{figure}

\begin{Theorem} \label{thm:embed}
The map $\psi:b \mapsto M_b$ from $\SSYT_n(\lambda)$ to $MS_n$ is a weak embedding of crystals.
\end{Theorem}

\begin{proof}
We must show that, for all $b \in \SSYT_n(\lambda)$ and all $i$,
$M_{e_i(b)} = e_i M_b$,
where on the left $e_i$ is calculated as in \S\ref{ss:yt}, and on the right as in \S\ref{ss:MS}. 
Let $S_i^{YT}(b)$ denote the string of brackets from \S\ref{ss:yt} and $S_i^{MS}(M_b)$ denote the string of brackets from \S\ref{ss:MS}. It is immediate from the definitions that these differ only in the following ways:
\begin{itemize}
\item  $S_i^{YT}(b)$ may have some ``)" corresponding to ``$i-1$" in row i, which are all at the left end of the string. These are not present in $S_i^{MS}(M_b)$.

\item $S_i^{MS}(M_b)$ may have canceling pairs of brackets corresponding to pairs of segments of the same length, say $[h,i]$ canceling $[h-1,i-1]$. It can happen that these correspond to an $i$ directly above an $i-1$ in $b$, in which case this pair is not present in $S_i^{YT}(b)$. 
\end{itemize}
Neither of these changes affect the uncanceled ``)." 
\end{proof}

\section{Crystal isomorphism from the PBW basis to multisegments} \label{sec:pbwtoms}
For this section, all PBW bases are with respect to the reduced expression
\begin{equation} \label{eq:re} {\bf i}= (1,2,\cdots, n,1,2 \cdots (n-1) \cdots 1,2,3,1,2,1).
\end{equation}
Recall that the corresponding order on positive roots is
$$\beta_1= \alpha_1, \beta_2= \alpha_1+\alpha_2, \ldots, \beta_{n}= \alpha_1 + \cdots+\alpha_n, \beta_{n+1}= \alpha_2, \ldots, \beta_N= \alpha_n.$$  
To each positive root $\beta= \alpha_i+ \cdots+ \alpha_j$ associate the segment $[\beta]=[i,j]$.

\subsection{The isomorphism}

\begin{Definition} \label{def:PBW-MS-isom}
Let $\Phi$ be the map from PBW bases to multisegments that takes\break $F_{\beta_1}^{(a_1)} \cdots F_{\beta_N}^{(a_N)}$ to the multisegment with $a_j$ copies of each $[\beta_j]$. 
\end{Definition}

\begin{Theorem} \label{thm:PBW-MS-isom}
$\Phi$ is an isomorphism of $\mathfrak{sl}_{n+1}$ crystals. 
\end{Theorem}

The proof of Theorem~\ref{thm:PBW-MS-isom} will occupy the rest of this section. The idea is to use Proposition~\ref{prop:Saito-formula}, which says $T_i^{-1} \circ \tau$ acts on $B(\infty)$ as realized using PBW bases in the same way as $(e_i^*)^{max}f_i^N$ for large $N$ (here $\tau$ means set the first exponent to 0). We consider 
$T_n^{-1}\tau \cdots T_1^{-1}\tau$ acting on PBW monomials and the corresponding crystal operators acting on multisegments, and show that these agree as they must if $\Phi$ is to be an isomorphism (see Proposition~\ref{prop:sigma} below). The image of this map is PBW monomials/multisegments for $\mathfrak{sl}_{n} \subset \mathfrak{sl}_{n+1}$, and we can then use induction on rank. 

The next section is fairly technical. We encourage the reader to look at  Proposition~\ref{prop:sigma} and the example in Section~\S\ref{ss:example} before continuing.

\subsection{Some technical lemmas}

\begin{Definition} \label{def:edef}
Given $M \in MS_n$, define a multisegment $M^{(k)}$ and integer $a_k$ for each $1\leq k \leq n$ inductively by
\begin{itemize}
\item $a_1= \varepsilon_1(M)+\varepsilon_1^*(M)+\langle\wt(M), \alpha_1^\vee \rangle,$

\item $M^{(1)} = f_1^{a_1}(M),$ 

\item $a_2= \varepsilon_2(M^{(1)})+\varepsilon_2^*(M^{(1)})+\langle\wt(M^{(1)}), \alpha_2^\vee \rangle,$

\item $M^{(2)} = f_2^{a_2}f_1^{a_1}(M),$ 

\end{itemize}
and so on.
\end{Definition}

\begin{Lemma} \label{lem:ineq} Fix $M \in MS_n$ and $i \leq k \leq n$. 

\begin{enumerate}

\item \label{ptg} $M_{i-1,k-1}^{(k)}= M_{i,k}$ (for $ i \geq 2$) and

\item \label{pty} $M_{i,k}^{(k)} \geq \max_{1 \leq s \leq n-k} \left\{ \sum\limits_{r=1}^s M_{i+1,k+r} - \sum\limits_{r=1}^{s-1} M_{i,k+r} \right\}.$ That is, 
 $M_{i,k}^{(k)}$ is at least the number of uncanceled ``)" in the substring of $S_i^*(M)$ consisting of those brackets that correspond to segments of length at least the length of $[i,k]$.
  
 \end{enumerate}
\end{Lemma}

\begin{proof} 
We proceed by induction on $k$. 
By Lemma~\ref{lem:bracket-count}, $a_1 = ur_1(M)+ur_1^*(M)$. So $f_1$ is applied to $M$ until there are no uncanceled ``)" in $S_1$, and then applied $ur_1^*(M)$ more times, creating $\young(1)$\;s and ``("s all the way to the left. This proves \eqref{pty} for $k=1$, and \eqref{ptg} is vacuous in this case.

Now assume the result holds for some $k$ and all $i \leq k$.
Look at $S_{k+1}(M^{(k)})$:

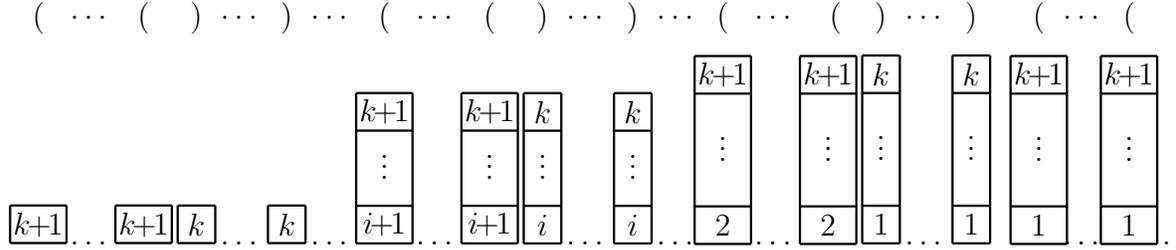
\[
\begin{tikzpicture}[scale=0.5]

\node at (0,.5){$\yywa{ $k\!\!+\!\!1$}$};
\node at (0,6){(};
\node at (1.4,0){$\dots$};
\node at (1.4,6){$\dots$};
\node at (2.8,.5){$\yywa{$k\!\!+\!\!1$}$};
\node at (2.8,6){(};
\node at (4.2,.5){$\yyta{$k$}$};
\node at (4.2,6){)};
\node at (5.4,0){$\dots$};
\node at (5.4,6){$\dots$};
\node at (6.6,.5){$\yyta{$k$}$};
\node at (6.6,6){)};
\node at (7.8,0){$\dots$};
\node at (7.8,6){$\dots$};
\node at (9.2,2){$\yywc{$i\!\!+\!\!1$}{$k\!\!+\!\!1$}$};
\node at (9.2,6){(};
\node at (10.6,0){$\dots$};
\node at (10.6,6){$\dots$};
\node at (12,2){$\yywc{$i\!\!+\!\!1$}{$k\!\!+\!\!1$}$};
\node at (12,6){(};
\node at (13.4,2){$\yytc{$i$}{$k$}$};
\node at (13.4,6){)};
\node at (14.6,0){$\dots$};
\node at (14.6,6){$\dots$};
\node at (15.8,2){$\yytc{$i$}{$k$}$};
\node at (15.8,6){)};
\node at (17,0){$\dots$};
\node at (17,6){$\dots$};
\node at (18.2,2.5){$\yywd{2}{$k\!\!+\!\!1$}$};
\node at (18.2,6){(};
\node at (19.6,0){$\dots$};
\node at (19.6,6){$\dots$};
\node at (21,2.5){$\yywd{2}{$k\!\!+\!\!1$}$};
\node at (21.2,6){(};
\node at (22.4,2.5){$\yytd{1}{$k$}$};
\node at (22.4,6){)};
\node at (23.6,0){$\dots$};
\node at (23.6,6){$\dots$};
\node at (24.8,2.5){$\yytd{1}{$k$}$};
\node at (24.8,6){)};
\node at (26.6,2.5){$\yywd{1}{$k\!\!+\!\!1$}$};
\node at (26.6,6){(};
\node at (27.8,0){$\dots$};
\node at (27.8,6){$\dots$};
\node at (29,2.5){$\yywd{1}{$k\!\!+\!\!1$}$};
\node at (29,6){(};

\node at (30,0) {.};

\end{tikzpicture}
\]

\noindent By the induction hypothesis, $M_{i,k}^{(k)} \geq M_{i+1,k+1}$ for all $i$.
Therefore all the ``(" over $[i+1, k+1]$ segments cancel ``)" over $[i,k]$ segments, and no other ``)" are canceled. 
Again using Lemma~\ref{lem:bracket-count}, $a_{k+1}= ur_{k+1}(M^{(k)})+  ur^*_{k+1}(M^{(k)})$, so applying $f_{k+1}^{a_{k+1}}$ changes all but $M_{i+1,k+1}^{(k)}$ of the $[i,k]$ segments to $[i,k+1]$ segments, establishing \eqref{ptg} for $k+1$. Furthermore, assuming  $i \leq k$, this creates exactly $M_{i,k}^{(k)} - M_{i+1,k+1}^{(k)}$ many $[i, k+1]$'s. Adding the original number of these segments,
\begin{align}
M_{i,k+1}^{(k+1)} &= M_{i,k+1} + M_{i,k}^{(k)} - M_{i+1,k+1} \nonumber \\
&\geq \max_{1 \leq s \leq n-k} \left\{ \sum\limits_{r=1}^s M_{i+1,k+r} - \sum\limits_{r=1}^{s-1} M_{i,k+r} \right\} - M_{i+1,k+1} + M_{i,k+1}  \label{ind} \\
&= \max_{1 \leq s \leq n-k} \left\{ \sum\limits_{r=2}^s M_{i+1,k+r} - \sum\limits_{r=2}^{s-1} M_{i,k+r} \right\}, \nonumber 
\end{align}
where \eqref{ind} holds by the induction hypothesis. After shifting indices this gives \eqref{pty} for $k+1$. Statement \eqref{pty} for $M_{k+1,k+1}^{(k+1)}$ holds as in the case $k=1$. 
\end{proof}

\begin{Corollary}\label{cor:first-half} $M^{(n)}$ is the multisegment obtained from $M$ by
 \begin{enumerate}
  \item \label{ewr1} removing all $[1, i]$ segments for each $i$,
 \item \label{ewr2} replacing each segment of the form $[i+1,k+1]$ for $i\geq 1$ by $[i, k]$, and
 \item \label{ewr3} adding some number of $[i,n]$ segments for each $i$.
 \end{enumerate}
\end{Corollary}

\begin{proof}
By Lemma~\ref{lem:ineq}\eqref{ptg} the number of segments $[i,k]$ for each $i \leq k<n$ is given by \eqref{ewr1} and \eqref{ewr2}, since applying $f_\ell$ for $\ell>k$ does not change the number of $[i,k]$ segments. 
Certainly $[i,n]$ segments can only be created, not destroyed, so \eqref{ewr3} holds as well. 
\end{proof}

\begin{Proposition}\label{prop:sigma} Fix a multisegment $M$. Then
$\sigma_{n} \dots \sigma_{2} \sigma_{1}(M)$ is the multisegment obtained by
\begin{enumerate}
\item Removing all $[1, i]$ segments
\item Shifting all remaining segments down by 1; i.e., for every $[i, j]$ in $M$, there is a $[i-1, j-1]$ in $\sigma_{n} \dots \sigma_{2} \sigma_{1}(M)$.

\end{enumerate}
\end{Proposition}

\begin{proof}
By Proposition~\ref{prop:comb-characterization}\eqref{ccc1} and Definition~\ref{def:crystal}\eqref{cry3}, for $i\neq j$ we have $(e_i^*)^{max}f_j(M)=f_j(e_i^*)^{max}(M)$. Thus 
$$\sigma_{n} \dots \sigma_{2} \sigma_{1} (M)= (e_n^*)^{max} \dots (e_2^*)^{max}(e_1^*)^{max}f_n^{a_n} \dots f_2^{a_2}f_1^{a_1}(M), $$
where $a_i$ is as in Definition~\ref{def:edef}. By Corollary~\ref{cor:first-half}, $f_n^{a_n} \dots f_2^{a_2}f_1^{a_1}(M)$ is correct except that there may be extra segments $[i,n]$ for each $i$. 
It remains to show that $$(e_n^*)^{max} \dots (e_2^*)^{max}(e_1^*)^{max}$$ just deletes all the $[i,n]$.

We proceed by induction on $i$, proving that each $(e_i^*)^{max}$ changes all $[i, n]$ segments into $[i+1, n]$ segments and does nothing else. 
For $i=1$, we have
\[
\begin{tikzpicture}[scale=0.5]
\node at (-3,1.5){$S_1^*(M^{(n)}):$};
\node at (0,1.5){$\yyta{1}$};
\node at (0,0){(};
\node at (1.2,1){$\dots$};
\node at (1.2,0){$\dots$};
\node at (2.4,1.5){$\yyta{1}$};
\node at (2.4,0){(};
\node at (3.6,1.5){$\yyta{2}$};
\node at (3.6,0){)};
\node at (4.8,1){$\dots$};
\node at (4.8,0){$\dots$};
\node at (6,1.5){$\yyta{2}$};
\node at (6.1,0){)};
\node at (7.2,1){$\dots$};
\node at (7.2,0){$\dots$};
\node at (8.4,3){$\yytc{1}{$j$}$};
\node at (8.4,0){(};
\node at (9.6,1){$\dots$};
\node at (9.6,0){$\dots$};
\node at (10.8,3){$\yytc{1}{$j$}$};
\node at (10.8,0){(};
\node at (12.3,3){$\yywc{2}{$j\!\!+\!\!1$}$};
\node at (12.3,0){)};
\node at (13.7,1){$\dots$};
\node at (13.7,0){$\dots$};
\node at (15,3){$\yywc{2}{$j\!\!+\!\!1$}$};
\node at (15,0){)};
\node at (16.4,1){$\dots$};
\node at (16.4,0){$\dots$};
\node at (17.6,3.5){$\yytd{1}{$n$}$};
\node at (17.6,0){(};
\node at (18.8,1){$\dots$};
\node at (18.8,0){$\dots$};
\node at (20,3.5){$\yytd{1}{$n$}$};
\node at (20,0){(};
\end{tikzpicture}
\]

\noindent By Corollary~\ref{cor:first-half}, the number of $[1, j]$'s in $M^{(n)}$ is equal to the number of $[2, j+1]$'s in $M$. Thus the ``)'' in $S_1^*(M)$ would be enough to cancel all the $[1,j]$ in $S_1^*(M^{(n)})$. Application of the various $f_k^{a_k}$ can only have made these $[2,j+1]$ segments longer, and hence moved these ``)" to the right, but not past the ``(" coming from any $[1,n]$ segment. It follows that the ``(" corresponding to intervals $[1,j]$ for $j<n$ are in fact all canceled in $S^*_1(M^{(n)})$, so these segments are not affected by $(e_1^*)^{max}$. Certainly the $[1,n]$ intervals are all changed to $[2,n]$ by $(e_1^*)^{max}$, so the $i=1$ case has been established. 

Now fix $k \geq 2$ and assume that the claim holds for all $i<k$. 
By the induction hypothesis, the difference between $S_k^*((e_{k-1}^*)^{max} \dots (e_1^*)^{max}M^{(n)})$ and $S_k^*(M^{(n)})$ is only that the first string has some extra ``(" corresponding to $[k,n]$ segments. By the same argument as the base case, the uncanceled ``(" in $S_k^*(M^{(n)})$ are precisely those below $[k,n]$ segments, so this remains true for $S_k^*((e_{k-1}^*)^{max} \dots (e_1^*)^{max}M^{(n)})$. Hence $(e_k^*)^{max}$ just changes all $[k,n]$ segments to $[k+1,n]$'s. 

At the final step, $e_n^*$ clearly deletes all the $[n,n]$ segments. 
\end{proof}

\subsection{Proof of Theorem~\ref{thm:PBW-MS-isom}}

For each $n$ let $PBW_n$ denote the crystal of PBW monomials corresponding to the reduced expression \eqref{eq:re}. $PBW_n \simeq B(\infty) \simeq MS_n$ and $B(\infty)$ is connected, so there is a unique crystal isomorphism $\phi_n:PBW_n \to MS_n$ for each $n$. We need to show that $\phi_n=\Phi_n$. 
We proceed by induction on rank, the $\mathfrak{sl}_2$ case being trivial. So, fix $n \geq 2$ and assume $\phi_{n-1}=\Phi_{n-1}$. 

Let $\phi_n(F_1^{(a_1)} F_{\alpha_1+\alpha_2}^{(a_2)} \cdots F_n^{(a_N)}) = \left\lbrace \young(1)_{\; c_1} \cdots \young(n)_{\; c_N} \right\rbrace.$ We need to show that $c_i=a_i$ for all $i$. By Proposition~\ref{prop:Saito-formula},
\begin{equation}\label{ind-commute} 
\phi_n(T_n^{-1}\tau \cdots T_1^{-1}\tau(F_1^{(a_1)} \cdots F_n^{(a_N)})) = (\sigma_n \dots \sigma_1 \circ \phi_n)(F_1^{(a_1)} \cdots F_n^{(a_N)}),
\end{equation}
where on the right side $\sigma_i$ is calculated as $(e_i^*)^{max} f_i^N$ for large $N$. 

There is a natural copy of $PBW_{n-1} \subset PBW_n$; 
the monomials where the exponent of $F_{\alpha_j + \cdots +\alpha_n}$ is zero for all $j$. 
The image of $\phi_{n}|_{PBW_{n-1}}$ consists of exactly those multisegments with no segments of the form $[i,n]$ for any $i$, which is naturally identified with $MS_{n-1}$. Certainly
$\phi_n|_{PBW_{n-1}}$ is still a crystal isomorphism, so by induction is equal to $ \phi_{n-1}$. Thus the left side of \eqref{ind-commute} is:
\begin{align*}
\phi_n(T_n^{-1}\tau \cdots T_1^{-1}\tau(F_1^{(a_1)} \cdots{} &F_n^{(a_N)})) = \phi_n(F_1^{(a_{n+1)}} \cdots F_{n-1}^{(a_N)}) \text{ by Lemma~\ref{lem:braid}} \\
&= \phi_{n-1}(F_1^{(a_{n+1})} \cdots F_{n-1}^{(a_N)}) \\
&= \left\lbrace \yyta{1}_{\; a_{n+1}} \cdots \yywa{$n\!\!-\!\!1$}_{\; a_N} \right \rbrace \text{ by the induction hypothesis.}
\end{align*}

On the right side we have:
\begin{align*}
\sigma_n \dots \sigma_1(\phi(F_1^{(a_1)} \cdots F_n^{(a_N)})) &= \sigma_n \dots \sigma_1 \left( \left\lbrace \young(1)_{\; c_1} \cdots \young(n)_{\; c_N} \right\rbrace \right) \\
&= \left\lbrace \yyta{1}_{\; c_{n+1}} \cdots \yywa{$n\!\!-\!\!1$}_{\; c_N} \right\rbrace \text{ by Proposition~\ref{prop:sigma}}.
\end{align*}
Hence $c_i = a_i$ for $i>n$. Also
$$
\wt(F_1^{(a_1)} \cdots F_n^{(a_N)}) = \wt(\phi(F_1^{(a_1)} \cdots F_n^{(a_N)})) 
$$
implies
$$
-(a_1\beta_1+ \dots + a_N\beta_N) = -( c_1\beta_1+ \dots + c_n\beta_n+ a_{n+1}\beta_{n+1} + \dots + a_N\beta_N) ,
$$
from which it follows that
$$
a_1\beta_1 + \dots + a_n\beta_n = c_1\beta_1+ \dots + c_n\beta_n,
$$
since $a_i=c_i$ for all other $i$. But $\beta_1 = \alpha_1, \beta_2 = \alpha_1+\alpha_2, \dots, \beta_n = \alpha_1+ \dots + \alpha_n$ are linearly independent, so it follows that $c_i=a_i$ for $i \leq n$ as well. \qed

\subsection{An example}  \label{ss:example}
The main difficulty in proving Theorem~\ref{thm:PBW-MS-isom} is establishing Proposition~\ref{prop:sigma}.
As discussed in the proof of that proposition, 
\begin{equation}
\label{ewq}  \sigma_n \cdots \sigma_1 M= (e_n^*)^{max} \dots (e_2^*)^{max}(e_1^*)^{max}f_n^{a_n} \dots f_2^{a_2}f_1^{a_1}(M).
\end{equation}
Here we go through the reasons why the right hand side of \eqref{ewq} has the desired effect on
\[
\label{exp:sigma}
M= \young(1) \; \young(2) \; \young(2) \; \young(3) \; \young(2,1) \; \young(2,1) \; \young(3,2) \; \young(4,3) \; \young(4,3) \; \young(4,3,2) \; \young(5,4,3,2) \;.
\]
The strings of brackets for $i=1$ are:
\[
\begin{tikzpicture}[scale=0.58]
\node at (0.5,4.5){$S_1(M):$};
\node [anchor=south west] at (0,0){$\young(3)$};
\node [anchor=south west] at (1,0){$\young(2)$};
\node [anchor=south west] at (2,0){$\young(2)$};
\node [anchor=south west] at (3,0){$\young(1)$};
\node [anchor=south west] at (3.25,4){(};
\node [anchor=south west] at (4,0){$\young(4,3)$};
\node [anchor=south west] at (5,0){$\young(4,3)$};
\node [anchor=south west] at (6,0){$\young(3,2)$};
\node [anchor=south west] at (7,0){$\young(2,1)$};
\node [anchor=south west] at (8,0){$\young(2,1)$};
\node [anchor=south west] at (9,0){$\young(4,3,2)$};
\node [anchor=south west] at (10,0){$\young(5,4,3,2)$};
\end{tikzpicture}
\;\;\;
\begin{tikzpicture}[scale=0.58]
\node at (-1.3,0.7){$S_1^*(M):$};
\node [anchor=south west] at (0,1){$\young(1)$};
\node [anchor=south west] at (.25,0){(};
\node [anchor=south west] at (1,1){$\young(2)$};
\node [anchor=south west] at (1.25,0){)};
\node [anchor=south west] at (2,1){$\young(2)$};
\node [anchor=south west] at (2.25,0){)};
\node [anchor=south west] at (3,1){$\young(3)$};
\node [anchor=south west] at (4,1){$\young(2,1)$};
\node [anchor=south west] at (4.25,0){(};
\node [anchor=south west] at (5,1){$\young(2,1)$};
\node [anchor=south west] at (5.25,0){(};
\node [anchor=south west] at (6,1){$\young(3,2)$};
\node [anchor=south west] at (6.25,0){)};
\node [anchor=south west] at (7,1){$\young(4,3)$};
\node [anchor=south west] at (8,1){$\young(4,3)$};
\node [anchor=south west] at (9,1){$\young(4,3,2)$};
\node [anchor=south west] at (9.25,0){)};
\node [anchor=south west] at (10,1){$\young(5,4,3,2)$};
\node [anchor=south west] at (10.25,0){) \; .};
\draw[line width=0.01cm] (.3,.15)--(2,.95);
\draw[line width=0.01cm] (4.4,.3)--(10,.9);
\end{tikzpicture} 
\]
By counting uncanceled ``(" we see that $\varepsilon_1(M)=1$ and $\varepsilon_1^*(M)=0$. We can also calculate $\langle wt(M), \alpha_1^{\vee} \rangle=1$, giving $a_1=2.$ As in Lemma~\ref{lem:bracket-count}, this is the number of uncanceled ``)"  in $S_1(M)$ plus the number of uncanceled ``)"  in $S^*_1(M)$. Applying $f_1^2$ creates two new $\young(1)$. So at the next step we get:
\[
\mbox{} \hspace{-0.4cm}
\begin{tikzpicture}[scale=0.58]
\node at (-0.3,4.6){$S_2(M^{(1)}) \hspace{-0.1cm} :$};
\node [anchor=south west] at (0,0){$\young(3)$};
\node [anchor=south west] at (1,0){$\young(2)$};
\node [anchor=south west] at (1.25,4){(};
\node [anchor=south west] at (2,0){$\young(2)$};
\node [anchor=south west] at (2.25,4){(};
\node [anchor=south west] at (3,0){$\young(1)$};
\node [anchor=south west] at (3.25,4){)};
\node [anchor=south west] at (4,0){$\young(1)$};
\node [anchor=south west] at (4.25,4){)};
\node [anchor=south west] at (5,0){$\young(1)$};
\node [anchor=south west] at (5.25,4){)};
\node [anchor=south west] at (6,0){$\young(4,3)$};
\node [anchor=south west] at (7,0){$\young(4,3)$};
\node [anchor=south west] at (8,0){$\young(3,2)$};
\node [anchor=south west] at (9,0){$\young(2,1)$};
\node [anchor=south west] at (9.25,4){(};
\node [anchor=south west] at (10,0){$\young(2,1)$};
\node [anchor=south west] at (10.25,4){(};
\node [anchor=south west] at (11,0){$\young(4,3,2)$};
\node [anchor=south west] at (12,0){$\young(5,4,3,2)$};
\draw[line width=0.01cm] (1.4,4.3)--(5,4.90);
\end{tikzpicture}
\;
\begin{tikzpicture}[scale=0.58]
\node at (1.7,0.7){$S_2^*(M^{(1)})  \hspace{-0.1cm} :$};
\node [anchor=south west] at (0,1){$\young(1)$};
\node [anchor=south west] at (1,1){$\young(1)$};
\node [anchor=south west] at (2,1){$\young(1)$};
\node [anchor=south west] at (3,1){$\young(2)$};
\node [anchor=south west] at (3.25,0){(};
\node [anchor=south west] at (4,1){$\young(2)$};
\node [anchor=south west] at (4.25,0){(};
\node [anchor=south west] at (5,1){$\young(3)$};
\node [anchor=south west] at (5.25,0){)};
\node [anchor=south west] at (6,1){$\young(2,1)$};
\node [anchor=south west] at (7,1){$\young(2,1)$};
\node [anchor=south west] at (8,1){$\young(3,2)$};
\node [anchor=south west] at (8.25,0){(};
\node [anchor=south west] at (9,1){$\young(4,3)$};
\node [anchor=south west] at (9.25,0){)};
\node [anchor=south west] at (10,1){$\young(4,3)$};
\node [anchor=south west] at (10.25,0){)};
\node [anchor=south west] at (11,1){$\young(4,3,2)$};
\node [anchor=south west] at (11.25,0){(};
\node [anchor=south west] at (12,1){$\young(5,4,3,2)$};
\node [anchor=south west] at (12.25,0){( \hspace{0.01cm} .};
\draw[line width=0.01cm] (3.4,.3)--(11,.90);
\end{tikzpicture}
\]
We see that $a_2=1$, and $f_2^{a_2}$ just changes the uncanceled $\young(1)$ to a $\young(2,1) \;$.
Continuing we find 
$a_3=3$, $a_4=2$, and $a_5=4$. Application of the operators gives
\[
\begin{tikzpicture}[scale=0.58]
\node at (-1.5, 1){$M^{(5)}=$};
\node [anchor=south west] at (0,0){$\young(2)$};
\node [anchor=south west] at (1,0){$\young(1)$};
\node [anchor=south west] at (2,0){$\young(1)$};
\node [anchor=south west] at (3,0){$\young(3,2)$};
\node [anchor=south west] at (4,0){$\young(3,2)$};
\node [anchor=south west] at (5,0){$\young(2,1)$};
\node [anchor=south west] at (6,0){$\young(5,4,3)$};
\node [anchor=south west] at (7,0){$\young(5,4,3)$};
\node [anchor=south west] at (8,0){$\young(5,4,3)$};
\node [anchor=south west] at (9,0){$\young(3,2,1)$};
\node [anchor=south west] at (10,0){$\young(4,3,2,1)$};
\node [anchor=south west] at (11,0){$\young(5,4,3,2)$};
\node [anchor=south west] at (12,0){$\young(5,4,3,2)$ \; .};
\end{tikzpicture}
\]
Notice that the segments in $M^{(5)}$ that do not contain 5 correspond exactly to the segments of $M$ that do not contain 1, but shifted down by one. This is the content of Corollary~\ref{cor:first-half}.

Next we must apply the $(e_i^*)^\text{max}$. At the first step we have:
\[
\begin{tikzpicture}[scale=0.58]
\node [anchor=south west] at (-3.5, 0){$S_1^*(M^{(5)})$:};
\node [anchor=south west] at (0,1){$\young(1)$};
\node [anchor=south west] at (.25,0){(};
\node [anchor=south west] at (1,1){$\young(1)$};
\node [anchor=south west] at (1.25,0){(};
\node [anchor=south west] at (2,1){$\young(2)$};
\node [anchor=south west] at (2.25,0){)};
\node [anchor=south west] at (3,1){$\young(2,1)$};
\node [anchor=south west] at (3.25,0){(};
\node [anchor=south west] at (4,1){$\young(3,2)$};
\node [anchor=south west] at (4.25,0){)};
\node [anchor=south west] at (5,1){$\young(3,2)$};
\node [anchor=south west] at (5.25,0){)};
\node [anchor=south west] at (6,1){$\young(3,2,1)$};
\node [anchor=south west] at (6.25,0){(};
\node [anchor=south west] at (7,1){$\young(5,4,3)$};
\node [anchor=south west] at (8,1){$\young(5,4,3)$};
\node [anchor=south west] at (9,1){$\young(5,4,3)$};
\node [anchor=south west] at (10,1){$\young(4,3,2,1)$};
\node [anchor=south west] at (10.25,0){(};
\node [anchor=south west] at (11,1){$\young(5,4,3,2)$};
\node [anchor=south west] at (11.25,0){)};
\node [anchor=south west] at (12,1){$\young(5,4,3,2)$};
\node [anchor=south west] at (12.25,0){) \;.};
\draw[line width=0.01cm] (.35,.35)--(12.8,.8);
\end{tikzpicture}
\]
There are no uncanceled ``(" so $\varepsilon_1^*=0$ and $(e_1^*)^{max}$ does nothing. Next,
\[
\begin{tikzpicture}[scale=0.58]
\node [anchor=south west] at (-5, 0){$S_2^*((e_1^*)^0 M^{(5)})$:};
\node [anchor=south west] at (0,1){$\young(1)$};
\node [anchor=south west] at (1,1){$\young(1)$};
\node [anchor=south west] at (2,1){$\young(2)$};
\node [anchor=south west] at (2.25,0){(};
\node [anchor=south west] at (3,1){$\young(2,1)$};
\node [anchor=south west] at (4,1){$\young(3,2)$};
\node [anchor=south west] at (4.25,0){(};
\node [anchor=south west] at (5,1){$\young(3,2)$};
\node [anchor=south west] at (5.25,0){(};
\node [anchor=south west] at (6,1){$\young(3,2,1)$};
\node [anchor=south west] at (7,1){$\young(5,4,3)$};
\node [anchor=south west] at (7.25,0){)};
\node [anchor=south west] at (8,1){$\young(5,4,3)$};
\node [anchor=south west] at (8.25,0){)};
\node [anchor=south west] at (9,1){$\young(5,4,3)$};
\node [anchor=south west] at (9.25,0){)};
\node [anchor=south west] at (10,1){$\young(4,3,2,1)$};
\node [anchor=south west] at (11,1){$\young(5,4,3,2)$};
\node [anchor=south west] at (11.25,0){(};
\node [anchor=south west] at (12,1){$\young(5,4,3,2)$};
\node [anchor=south west] at (12.25,0){( \; .};
\draw[line width=0.01cm] (2.35,.35)--(9.8,.8);
\end{tikzpicture}
\]
So $e_2^*$ is applied twice, deleting the $2$s at the bottom of the right-most segments. Notice that all ``(" corresponding to segments that do not contain a 5 are canceled; as discussed in the proof of Proposition~\ref{prop:sigma} this will always be the case, so no segments are changed except those containing $5$s. Certainly segments containing 5s always do change, and at the final step are deleted.
 We end up with
\[
\begin{tikzpicture}[scale=0.58]
\node [anchor=south west] at (-18.5, 0){$(e_5^*)^\text{max}(e_4^*)^\text{max}(e_3^*)^\text{max}(e_2^*)^\text{max}(e_1^*)^\text{max} f_5^{a_5}f_4^{a_4}f_3^{a_3} f_2^{a_2} f_1^{a_1} (M)=$};
\node [anchor=south west] at (0,0){$\young(1)$};
\node [anchor=south west] at (1,0){$\young(1)$};
\node [anchor=south west] at (2,0){$\young(2)$};
\node [anchor=south west] at (3,0){$\young(2,1)$};
\node [anchor=south west] at (4,0){$\young(3,2)$};
\node [anchor=south west] at (5,0){$\young(3,2)$};
\node [anchor=south west] at (6,0){$\young(3,2,1)$};
\node [anchor=south west] at (7,0){$\young(4,3,2,1)$};
\node [anchor=south west] at (8,0){$,$};
\end{tikzpicture}
\]
as predicted by Proposition~\ref{prop:sigma}.

\end{document}